\documentclass[12pt]{article}

\def\BBox{\vrule height 0.5em width 0.6em depth 0em}

\usepackage{amssymb}
\usepackage[thmmarks]{ntheorem}
\theoremheaderfont{\bfseries}
\theorembodyfont{\normalfont}
\theoremseparator{:}
\theoremsymbol{$\blacksquare$}
\newtheorem{thm}{Theorem}[section]
\newtheorem{exm}[thm]{Example}
\newtheorem{algm}[thm]{Algorithm}
\newtheorem{lem}[thm]{Lemma}
\newtheorem{prop}[thm]{Proposition}
\newtheorem{cor}[thm]{Corollary}
\newtheorem{rmk}[thm]{Remark}

\newtheorem*{proof}{Proof}

\newcommand{\RPlus}{\textit{R}_{+}}
\newcommand{\abs}[1]{\left\vert#1\right\vert}
\newcommand{\norm}[1]{\left\| #1 \right\|}
\newcommand{\seq}[1]{\left<#1\right>}
\newcommand{\set}[1]{\left\{#1\right\}}

\newcommand{\cprk}{\textit{cp-rank}}
\newcommand{\al}{\alpha}
\newcommand{\be}{\beta}

\newcommand{\eps}{\epsilon}

\begin{document}
\title{When does cp-rank equal rank?}

\author{Wasin So \thanks{Department of Mathematics, San Jose State University, San Jose, CA 95120, USA.
 Email: wasin.so@sjsu.edu.}
\and 
Changqing Xu \thanks{ Corresponding author.  
Department of Applied Mathematics,
Suzhou University of Science and Technology, Suzhou, 215009, China   Email: cqxurichard@mail.usts.edu.cn.}}
\date{}
\maketitle

\begin{abstract}
The problem of finding completely positive matrices with  equal cp-rank and rank is considered.
We give some easy-to-check sufficient conditions on the entries of a doubly nonnegative matrix
for it to be completely positive with equal cp-rank and rank. An algorithm is also presented to show its efficiency.  

\end{abstract}

\section{Introduction}

An $n\times n$ real matrix $A$ is called \textsl{completely positive} (CP) if there is some entrywise nonnegative matrix $B\in \textit{R}^{m\times n}$  such that $A$ can be factorized as $A=B^T B$. The minimum $m$ is called the \cprk of $A$ and is denoted by  \cprk$(A)$.  Completely positive matrices were used to deal with configuration problems in combinatorics \cite{h1}, and were also employed in statistics \cite{gw}. Shashua and Hazan  applied completely positive factorization method   to cluster data sets in computer vision \cite{sh}. Recent results (up to 2003) on completely positive matrices are presented in the book by Berman and Shaked-Monderer \cite{bs}. The latest results on CP matrices are about their geometric interpretation   by cone theory  and  are closely related to mathematical programming \cite{ab,b2,bad,ds}.

\indent Let $\textbf{R}$ be the set of real numbers, and  $\textbf{R}_+$ be the set of nonnegative real numbers. An $n\times n$ entrywise nonnegative matrix is called \textsl{doubly nonnegative} (DN) if it is also positive semi-definite (PSD). 
As usual, we denote $PSD_n$ the set of PSD matrices of order $n$,  $DN_n$ the set of DN matrices of
order $n\ge 1$, and $CP_n$ the set of  CP matrices of order $n\ge 1$. 

\indent A CP matrix is obviously a DN matrix, but the converse is not necessarily true in general. Among many sufficient conditions  for a DN matrix to be CP, we mention the one by Kaykobad using diagonally dominance.
 A diagonally dominant matrix $A=[a_{ij}]$ satisfies the condition  $|a_{ii}| \ge \sum_{j \ne i} |a_{ij}|$
for all $i$. Kaykobad  \cite{ka} proved that a  diagonally dominant DN matrix is indeed a  CP matrix with  \cprk $\le  $ number of nonzero
 entries above the diagonal + number of strictly diagonally dominant rows.

\indent A CP matrix is more than just a DN matrix. Among many necessary conditions, we include Proposition 1.1,
which is a consequence of a stronger result in \cite{djl}:
A DN matrix $A$ with a
triangle-free graph $G(A)$ is CP if and only if its comparison matrix $M(A)$ is positive semidefinite,
 in this case, \cprk$(A)=\max\{ rank(A), e(G(A))\}$, where $e(G(A))$ is the number of edges of $G(A)$.
 Recall that (i) for a symmetric matrix $A=[a_{ij}]$, we define $G(A)$ as an undirected graph on the vertex set $\{1,2,\ldots,n\}$, and for $i \ne j$,  $(i,j)$ is an edge if and only if $a_{ij} \ne 0$,
(ii) for nonnegative $A$, we define $M(A)=2D(A)-A$, where $D(A)$ is the diagonal part of $A$.
We provide a proof of Proposition 1.1 because it is different from the one in  \cite{djl}.

\begin{prop}
Let $A \in CP_n$ with $G(A)$ being a cycle and $n \ge 4$.
Then off-diagonal sum $\le$ diagonal sum, and  \cprk$(A) \ge n$.
\end{prop}

\begin{proof}
  Let $A=[a_{ij}]$. Since $G(A)$ is a cycle, by relabeling if necessary, we can assume that
the nonzero entries of $A$ above the diagonal are exactly  $a_{i,i+1}$ for $i=1,\ldots,n$.
(Here we adopt the convention that $n+1$ is 1 for any subscript.)
Now suppose that $A \in CP_n$ and so $A=B^TB$ where $B=[b_1\;b_2\;\cdots\;b_n]$
for some $b_i \in \textbf{R}_+^r$, and $r=$ \cprk$(A)$. Let  $b_i(k)$ be the $k$-th
entry of the vector $b_i$. Also denote $S_i$ the support of the vector $b_i$, i.e.,
$S_i = \{ k : b_i(k) \ne 0 \} \subseteq \{ 1, 2, \ldots, r\}$. Since  $B$ is a nonnegative
matrix and $A=B^TB$, we have $a_{ij} =0 $ if and only if $b_i^T b_j =0$ if and only if $S_i \cap S_j = \emptyset$.
Using the fact that $G(A)$ is a cycle and $n \ge 4$, we deduce that $S_1 \cap S_2, S_2 \cap S_3, \ldots,
S_{n-1} \cap S_n, S_n \cap S_1$ are pair-wise disjoint nonempty sets.
It follows that \cprk$(A)=r \ge |S_1 \cup \cdots \cup S_n| \ge |S_1 \cap S_2| + \cdots + |S_{n-1} \cap S_n| + | S_n \cap S_1| \ge n$.
Moreover, $( S_{i-1} \cap S_i ) \cup ( S_i \cap S_{i+1} ) = S_i$
because $S_i \cap S_j = \emptyset $ if $j \ne i-1, i+1$.
Now for nonzero off-diagonal entries
\[ a_{i,i+1} = b_i^T b_{i+1} = \sum_{k \in S_i \cap S_{i+1}} b_i(k) b_{i+1}(k)
\le \frac{1}{2} \sum_{k \in S_i \cap S_{i+1}} [ b_i(k)^2 + b_{i+1}(k)^2 ]. \]
Consequently, the conclusion follows from the computation below:
\begin{eqnarray*}
a_{12} +a_{23} + \cdots + a_{n-1,n} + a_{n1} &\le& \frac{1}{2} \sum_{i=1}^n \sum_{k \in S_i \cap S_{i+1}} [b_i(k)^2 + b_{i+1}(k)^2] \\
&=& \frac{1}{2} \sum_{i=1}^n \sum_{k \in ( S_{i-1} \cap S_i ) \cup ( S_i \cap S_{i+1}] ) } b_i(k)^2 \\
&=& \frac{1}{2} \sum_{i=1}^n \sum_{k \in  S_i } b_i(k)^2 \\
&=&   \frac{1}{2} \sum_{i=1}^n  a_{ii} \\
&=& \frac{1}{2} ( a_{11}  + \cdots + a_{nn} ). \; 
\end{eqnarray*}  
\end{proof} 
  
\noindent Note that Proposition 1.1 does not hold for $n<4$: 
A CP matrix $A$ of order 3 with $G(A)$ as a triangle can have \cprk$(A)=2$ or even 1.
For example, if $A=J$ is an all-ones matrix of order 3, then $A$ is CP with \cprk$(A)=1$.

A necessary and sufficient condition for a DN matrix to be CP is given in \cite{x},
where a geometric description of $CP_n$  using convex cone theory is given.
        However, this necessary and sufficient condition is not convenient to be employed
       for determining whether a DN matrix of  order greater than 4 is CP or not. Therefore, there is
        a need to search for other (easy-to-check) conditions.

 For any $A \in CP_n$, it is known that $rank(A) \le$ \cprk$(A)$. Generally the cp-rank of a CP matrix can be very larger than its rank.  It is interesting to determine when the equality holds. This problem was first put forward by N. Shaked-Monderer (\cite{N1},\cite{N2}) . She also proved that a nonnegative matrix generated by a Soules matrix is CP with cp-rank equal to the rank \cite{N2}. Note that a Soules matrix $S\in \textit{R}^{n\times n}$ is an orthogonal matrix whose first column is positive, and that $SDS^T\ge 0$ for each nonnegative diagonal matrix 
 $D= diag(d_1,\ldots,d_n)$ where $d_1\ge \ldots \ge d_n$. Another example for the equality $\cprk(A)=rank(A)$ is when the graph associated with $A$ is a cycle \cite{X01}.  
 
 Obviously, for a diagonal CP matrix $A$, we have $rank(A) =\cprk (A)$.  However, no characterization is known (yet) for a general case.  \\
 \indent  The reason for us to consider this case is that this is an extreme case for a DN matrix to be CP when the geometric explanation can be applied. Specifically, this is equivalent to the following geometric problem:\\
 \indent  \emph{ Given a set of vectors $V=\set{v_1,\ldots,v_n}$ in vector space $\textit{R}^r$.  When can $V$ be rotated into the positive orthant of  $\textit{R}^r$ if  the angle between any pair of vectors in $V$ is smaller than or equal to $\pi/2$?} \\
 \indent Up to the best of our knowledge, this is still an open problem in geometry.  So it deserves us to investigate it.  The characterization of a CP matrix without the restriction on its \cprk is surely more complicate than this situation. But it can also be  reformulated into a geometric problem where the rotation shall be replaced by a transformation between two different vector spaces with different dimensions. Though the transformation between these the two vector spaces with different dimensions still possesses the orthogonality (it is isometric), the difficulty of the problem in this situation lies in that we even have no idea at all about the dimension of the space of the image of this transformation (there are some results concerning the upper bound of the dimensions of this image space).\\ 
\indent   We denote by $CP(n,r)$ the set consisting of all $n \times n$ CP matrices $A$ with $\cprk(A) = rank(A) = r$, and $DN(n,r)$
 the set of all $n \times n$ DN matrices with $rank(A)=r$,  $PSD(n,r)$ the set of $n \times n$ PSD (positive semidefinite) matrices with $rank(A)=r$. Note that all these three sets are empty when $r>n$ and consist of only zero matrix when $r=0$. Thus we always assume throughout the paper that $0<r\le n$.\\
 \indent It is obvious that $CP(n,r) \subseteq DN(n,r) \subseteq PSD(n,r)$ for all $n$ and $r$.  There are a number of known sufficient conditions for a DN matrix to be CP with equal cp-rank and rank. Most of them are based on small order, low rank, or special zero pattern. For example, it is shown that a Soules matrix $A$ is always in $CP(n,r)$, where $A$ is defined as the form $A=RDR^T$ where $R\in \textit{R}^{n\times n}$ is an orthogonal matrix and $D=diag(d_1,\ldots,d_n)$ is nonnegative diagonal whose diagonal elements are in decreasing order, i.e., $d_1\ge \ldots \ge d_n\ge 0$. \\ 
\indent   It is obvious that $CP(n,0)= DN(n,0), CP(n,1)= DN(n,1)$ because the outer product of two vectors is nonnegative if and only if both vectors can be set to be entrywise nonnegative. The equality $CP(n,2)= DN(n,2)$  was first observed and proved by Gray and Wilson \cite{gw}. An alternative to see this result is that vectors in $\textbf{R}^2$ with nonnegative pairwise  inner products can be simultaneously rotated to the first quadrant (see Corollary 2.4). \\
\indent   For $r\ge 3$,   the equality $CP(n,r)= DN(n,r)$ does not hold for general $n$ (see Example 1.2). 
     However, we have $CP(3,3)=DN(3,3)$. A proof of this result can be found in \cite[page 140]{bs}. 
     So far we have $CP(n,r)=DN(n,r)$ for $0 \le r \le n \le 3$. \\

\indent   On the other hand, if $A \in DN(n,r)$ and $G(A)$, the associated graph of $A$, is a tree,  then $A \in CP(n,r)$ \cite{bh}. 
      The next example  shows that the conditions of small order, low rank, and special zero pattern are essential for having equal   \cprk and rank.

 \begin{exm}
     Consider the matrix
      \[A =  \left[ \begin{array}{cccc} 2& 1&0&1\\1&2&1&0\\0&1&2&1\\ 1&0&1&2 \end{array} \right].\]
      Then $A \in DN(4,3)$. But $A \not \in CP(4,3)$, otherwise, $3= rank(A) =$ \cprk$(A) \ge 4$, a contradiction!
      The last inequality follows from the fact that $G(A)$ is a cycle of length 4 and Proposition 1.1.
      Indeed, $A \in CP_4$ with  \cprk$(A)=4$ by Kaykobad's result.
 \end{exm}

    \noindent Therefore, for bigger order, higher rank, or arbitrary zero pattern, it needs additional condition
    for a DN matrix to be CP with equal cp-rank and rank.

\section{Geometric Approach}

In this section,  we give a sufficient condition on the entries of a doubly nonnegative matrix $A$ of rank $r$ for $A$
to be completely positive with \cprk  equals $r$. The proof is based on the following consequence of Cauchy-Schwarz inequality:
 \[  ( z_1+\cdots +z_k )^2 \le k ( z_1^2+\cdots+z_k^2 )    \]
where $z_1, \ldots, z_k$ are real numbers.
Let  $e$ be the  vector in  $\textbf{R}^r$ with all entries equal to 1, and $\seq{\cdot, \cdot}$ be the standard inner product in $\textbf{R}^n$ where $n$ is any positive integer, and $\| \cdot \|$ be the Euclidean norm of a vector.   The next lemma shows that any vector with a small angle with $e$ also has nonnegative entries. 

\begin{lem}
Let $z \in  \textbf{R}^r$ be such that $\seq{z,e}\ge \sqrt{\frac{r-1}{r}} \| z \| \| e \|$.
Then $z \in  \textbf{R}_+^r$.
\end{lem}

\begin{proof}
Without loss of generality, we assume 
$z=[ z_1, z_2, \ldots, z_r ]^T$ where $z_1 \ge z_2 \ge \cdots \ge z_r$. Then the hypothesis is
\[ z_1 + z_2 + \cdots + z_r \ge \sqrt{(r-1)(z_1^2+ \cdots + z_r^2)} \ge 0, \]
and so
\[ (z_1 + z_2 + \cdots + z_r)^2 \ge  (r-1)(z_1^2+ \cdots + z_r^2).\]
Now suppose to the contrary that there exists $1 \le k \le r-1$ such that $z_1 \ge z_2 \ge z_k \ge 0 > z_{k+1} \ge \cdots \ge z_r$.
Hence, $z_1 + z_2 + \cdots + z_k \ge 0$ and $-(z_{k+1} + \cdots +z_r) >0$. It follows from the inequality above that
\begin{eqnarray*}
  (z_1+\cdots + z_k +z_{k+1} + \cdots +z_r )^2
& \le & ( z_1+\cdots + z_k )^2 + (z_{k+1} + \cdots +z_r)^2 \\
& \le & k(z_1^2+\cdots+z_k^2) + (r-k) (z_{k+1}^2+\cdots+z_r^2).
\end{eqnarray*}
Combining with the hypothesis, we have
\[ (r-1)(z_1^2 + \cdots + z_r^2 ) \le  k(z_1^2+\cdots+z_k^2) + (r-k) (z_{k+1}^2+\cdots+z_r^2), \]
and so
\[ (r-1-k) (z_1^2+\cdots+z_k^2) + (r-1-r+k) (z_{k+1}^2+\cdots+z_r^2) \le 0. \]
Because of $z_{k+1} <0 $ and $z_1 + z_2 + \cdots + z_r \ge 0$, it follows that
$r-1-k=0$ and $r-1-r+k=0$, i.e., $k=1$ and $r=2$.
Consequently,  we have $z_1 \ge  0 > z_2$, and so $z_1 z_2 < 0$.
On the other hand, the hypothesis gives $z_1+z_2 \ge \sqrt{z_1^2 + z_2^2}$,
and so $2 z_1 z_2 \ge 0$, a contradiction! 
\end{proof}

\begin{exm}
This example shows that the parameter $\sqrt{\frac{r-1}{r}}$ is optimal.
For  $c < \sqrt{\frac{r-1}{r}}$, choose small $\eps >0$ such that
$-\eps +\sqrt{r-1} \sqrt{1-\eps^2} > c \sqrt{r}$. Take
\[ z=[ -\eps \sqrt{r-1}, \sqrt{1-\eps^2}, \cdots, \sqrt{1-\eps^2} ] \notin \textbf{R}_+^r  \]
and so
\begin{equation}
\seq{z,e} = - \eps \sqrt{r-1} + (r-1) \sqrt{1-\eps^2} =\sqrt{r-1} ( -\eps+ \sqrt{r-1} \sqrt{1-\eps^2}) > c \norm{z} \norm{e}
\end{equation}
\end{exm}

\begin{lem}
Given vectors $\beta_i \in \textbf{R}^r$.
If there exists a nonzero vector $x \in \textbf{R}^r$ such that
\[ \seq{\be_i, \; x} \ge \sqrt{\frac{r-1}{r}}  \norm{\be_i}\norm{x}  \]
for all $i$ then there exists an orthogonal matrix $Q$ such that $Q \beta_i \ge 0$, i.e.,
$Q\beta_i$ is a nonnegative vector.
\end{lem}

\begin{proof}
Case 1: $x$ is a positive multiple of $e$, i.e., $x=\frac{\|x\|}{\sqrt{r}} e$.
Then take $Q$ to be the identity matrix $I$, and we have $Qx=x=\frac{\|x\|}{\sqrt{r}} e$.

Case 2:  $x$ is not a positive multiple of $e$. Then $v=x - \frac{\|x\|}{\sqrt{r}} e \ne 0$.
Take $Q=I-\frac{2}{v^Tv} v v^T$. Hence, $Q$ is orthogonal since
\[  Q^T Q = Q^2 = \left( I-\frac{2}{v^Tv} v v^T \right)^2 = I - \frac{4}{v^Tv} v v^T + \frac{4}{(v^Tv)^2 } v v^Tv v^T =  I - \frac{4}{v^Tv} v v^T + \frac{4}{v^Tv} v v^T =I. \]
Moreover  $v^Tx = x^Tx - \frac{\|x\|}{\sqrt{r}} e^Tx $ and
\begin{eqnarray*} v^Tv &=& \left(x^T - \frac{\|x\|}{\sqrt{r}} e^T \right) \left(x - \frac{\|x\|}{\sqrt{r}} e \right) \\
&=& x^Tx-\frac{ 2 \|x\|}{\sqrt{r}} e^Tx + \frac{\|x\|^2}{r} e^Te \\
&=& 2 \left( x^Tx -  \frac{\|x\|}{\sqrt{r}} e^Tx \right) \\
&=& 2 v^Tx.
\end{eqnarray*}
And so 
\begin{eqnarray*} 
Qx & = & x-\frac{2}{v^Tv} v v^Tx \\
      & = & v+\frac{\|x\|}{\sqrt{r}} e -\frac{2}{v^Tv} v v^Tx \\
      & = & \left (1-\frac{2}{v^Tv} v^Tx \right ) v + \frac{\|x\|}{\sqrt{r}} e\\
      & = & \frac{\|x\|}{\sqrt{r}} e.
\end{eqnarray*}
Finally, in both cases, we have
\[ \frac{\langle Q \beta_i,  e \rangle}{\| Q \beta_i\| \| e\| } =
\frac{\langle Q \beta_i, \frac{\|x\|}{\sqrt{r}} e \rangle}{\| Q \beta_i\| \|\frac{\|x\|}{\sqrt{r}} e\| }
= \frac{\langle Q \beta_i,  Q x \rangle}{\| Q \beta_i\| \| Qx\| }
= \frac{\langle  \beta_i,   x \rangle}{\|  \beta_i \| \|  x  \|}
\ge \sqrt{ \frac{r-1}{r} },  \]
and so $Q \beta_i \ge 0$  by Lemma 2.1. 
\end{proof} 

\begin{cor} Let $A \in DN(n,2)$. Then $A \in CP(n,2)$.
\end{cor}

\begin{proof} Since $A$ is a positive semidefinite matrix of rank $2$, we have $A=B^TB$ for some 
$B=[v_1, v_2, \ldots ,v_n]\in \textbf{R}^{2\times n}$.  Moreover $A$ is nonnegative, and so 
$v_i^Tv_j =\seq{v_i, v_j}\ge 0$. Take $x\in \textbf{R}^2$ to be the angle bisector of the largest angle among all pairs of vectors.
Then $\seq{v_i, x}\ge \sqrt{\frac{1}{2}} \|v_i\| \|x\|$ for all $i$. By Lemma 2.3, there exists orthogonal $Q$  such that $Q v_i \ge 0$.
 Consequently, $A=B^TB=(QB)^T(QB)$ is CP with \cprk($A$)=2. 
 \end{proof}
 
\begin{thm}
Let $A=[a_{ij}] \in PSD(n,r)$, and denote $R_i$ the $i$-th row sum of $A$. If
\[ r R_i^2 \ge (r-1) a_{ii} (R_1+\cdots+ R_n)\] for
all $i$  then $A \in CP(n,r)$.
\end{thm}

\begin{proof}
Since $A$ is a positive semidefinite matrix of rank $r$, $A=B^TB$ where $B=[\be_1,\; \be_2, \ldots ,\be_n]$ and $\be_i \in \textit{R}^r$. Note that $\be_i = B e_i$ where $e_i$ is the $i$-th column of the $n \times n$ identity matrix. Take $x=\be_1 + \be_2 + \cdots +\be_n=Be_1+Be_2+\cdots+Be_n=
  B(e_1 + \cdots + e_n)=B e $.
Now \[ \langle \beta_i, x \rangle = \langle Be_i, Be\rangle = e_i^T B^T B e=e_i^T A e = R_i,\]
\[\| \beta_i \|^2 = \langle \beta_i, \beta_i \rangle=\langle Be_i, Be_i\rangle=e_i^T B^T B e_i=e_i^T A e_i = a_{ii},\]
and
\[ \| x \|^2= \langle x, x \rangle = \langle Be, Be  \rangle =e^T B^T B e = e^T A e = R_1+\cdots+R_n.\]
Hence, by hypothesis, \[ \frac{\langle \beta_i, x \rangle}{\| \beta_i \| \| x \|} = \frac{R_i}{\sqrt{a_{ii}} \sqrt{R_1+\cdots+R_n}} \ge \sqrt{\frac{r-1}{r}}. \]
By Lemma 2.3, there exists an $r \times r$  orthogonal matrix $Q$ such that $Q \beta_i \ge 0$,
and so $QB$ is an $r \times n$ nonnegative matrix.
Consequently, $A= B^T B = B^T Q^T Q B = (QB)^T (QB)$ is CP with  \cprk$(A) \le r$.
On the other hand, $r=rank(A) \le \cprk(A)$, so we have  $\cprk(A) =r$.  
\end{proof}

\begin{cor} Let $A \in PSD(n,3)$. If  $3 R_i^2 \ge 2 a_{ii} (R_1+\cdots R_n)$ for
all $i$, then $A \in CP(n,3)$.
\end{cor}

The next two examples show how  to use this sufficient condition to determine a given DN matrix of rank 3 to be CP with \textit{cp-rank} being 3.

\begin{exm}
   Consider matrix
\[ A = \left[ \begin{array}{cccc}
93  &  27   &  55   &  45 \\
    27   &  45    & 33  &   51\\
    55    & 33  &   41 &    43\\
    45    & 51   &  43  &   62
     \end{array} \right]. \]
It is easy to check that $A \in DN(4,3)$. Using the result of Maxfield and Minc \cite{mm} that any $4 \times 4$ DN matrix is CP (whose \cprk is less than or equal to 4), we know that $A$ is CP, but we don't know its exact \cprk.  However, by Corollary 2.6, we know that $A \in CP(4,3)$.
\end{exm}

\begin{exm}
Consider matrix
\[ A = \left[ \begin{array}{ccccc}
41   & 43  &  80  &  56 &   50 \\
    43  &  62  &  89  &  78 &   51\\
    80  &  89   &162  & 120  &  93\\
    56   & 78  & 120 &  104 &   62\\
    50  &  51  &  93   & 62  &  65
     \end{array} \right]. \]
Then $A \in DN(5,3)$ by simple calculations on $A$'s eigen-system. But the result of Maxifeld and Minc \cite{mm} cannot be applied. By Corollary 2.6, it can be easily checked that the sufficient condition is satisfied and so  $A \in CP(5,3)$.
\end{exm}

 \begin{rmk} Unfortunately, the sufficient condition in Theoem 2.5 is far from necessary. For example, the diagonal
 matrix $\left[ \begin{array}{cc} 100&0\\0&1 \end{array} \right] \in DN(2,2)=CP(2,2)$ but it fails the condition in Theorem 2.4.
 \end{rmk}
 
 \section{Nonnegative Equivalent Matrices}
 In this section, we first introduce \emph{nonnegative equivalent matrices}, which are a special kind of  doubly nonnegative matrices. We then show that a nonnegative equivalent DN matrix with rank three must be CP with its cp-rank equal to its rank. \\ 
\indent  Before we come to our main points, let us first state an already known result (see e.g. \cite{x}), by which we get the uniqueness of the minimal CP factorization for a matrix $A\in CP(n,r)$ in the sense of unitary transformation. Here by minimal CP factorization we mean the factorization $A=B^TB$ where $B\in \RPlus^{m\times n}$ with $m=\cprk(A)$. The following lemma is actually a special case of Lemma 2.1 of \cite{x}. We present here an alternative proof to this result in our situation. 
 
  \begin{lem}\label{le:le3.1}
 Let  $A=C^TC=B^TB$  where $B,C\in \textit{R}^{r\times n}$ and $r=rank(A)$. Then there exists an orthogonal matrix $Q\in \textit{R}^{m\times m}$ such that $C=QB$.
  \end{lem}
\begin{proof}
It is obvious that $rank(A)=rank(B)=rank(C)=r$. Now from the equality $B^TB=C^TC$, we get $B^TBB^T=C^TCB^T$. It follows that 
\[ B^T = C^TCB^T (BB^T)^{-1} \]
Here the $r\times r$ matrix $BB^T$ is invertible due to the fact $rank(BB^T)=rank(B)=r$. Thus we have $B=QC$ if we denote $Q=(BB^T)^{-1}BC^T$.  It suffices to prove that $Q$ is an orthogonal matrix. In fact, we have   
 \begin{eqnarray*}
 QQ^T = & [(BB^T)^{-1}BC^T][CB^T(BB^T)^{-1}]\\
            = & (BB^T)^{-1}B (C^TC) B^T(BB^T)^{-1} \\
            = &  (BB^T)^{-1}B (B^TB) B^T(BB^T)^{-1}=I_r
   \end{eqnarray*}
where $I_r$ is the identity matrix of order $r$. \\
\end{proof}

\begin{lem}\label{le: le3.2}
  For $r=1,2,3,4$, $DN(r,r)=CP(r,r)$. 
  \end{lem} 
  
\begin{proof}
It is easy to see that $DN(n,1)=CP(n,1), DN(n,2)=CP(n,2)$ and $DN(3,3)=CP(3,3)$ (see e.g. [2]).
Thus it suffices to prove $DN(4,4)=CP(4,4)$. It is obvious that $CP(4,4)\subseteq DN(4,4)$ by definition of $CP(n,r)$. Now we 
consider $ A \in DN(4,4)$. By Minc and Maxfield [\cite{mm}], $A$ is CP with  $\cprk(A) \le 4$,
hence $4 =rank(A) \le  \cprk(A) \le 4$. So $\cprk(A)=rank(A)=4$.  It follows that $ A \in CP(4,4)$ which implies 
$DN(4,4)\subseteq CP(4,4)$. Consequently, we have  $CP(4,4)=DN(4,4)$.
\end{proof} 
 
 \begin{exm}\label{ex: ex3.3}  This example shows that $CP(5,5) \ne DN(5,5)$. 
 We take $v_1 =\left[ \begin{array}{c} 1\\0\\1\\0\\0 \end{array} \right]$,
 $v_2 =\left[ \begin{array}{c} 1\\0\\0\\1\\0 \end{array} \right]$,
 $v_3 =\left[ \begin{array}{c} 1\\0\\0\\0\\1 \end{array} \right]$,
 $v_4 =\left[ \begin{array}{c} 0\\1\\1\\0\\0 \end{array} \right]$,
 $v_5 =\left[ \begin{array}{c} 0\\1\\0\\1\\0 \end{array} \right]$,
 $v_6 =\left[ \begin{array}{c} 0\\2\\0\\0\\1 \end{array} \right]$,
 and $A =\sum_{i=1}^6 v_i v_i^T =
 \left[ \begin{array}{ccccc} 3&0&1&1&1\\
 0&6&1&1&2\\1&1&2&0&0\\1&1&0&2&0\\1&2&0&0&2 \end{array} \right]$.
 Since $\det(A)=4 \ne 0$,
 and so $A \in DN(5,5)$.
 Let $k=cp\mbox{-}rank(A)$. Since $A=\sum_{i=1}^6 v_iv^T_i$, $k \le 6$ because $v_i \ge 0$.
 On the other hand, if $A=\sum_{i=1}^k b_ib^T_i$ for some $b_i \ge 0$.
 Then $G(b_ib^T_i)$ forms a clique in $G(A)$,  and $G(A)$ has $K_1$ and $K_2$ as the only cliques,
 hence $G(b_ib^T_i)$ is a vertex or an edge of $G(A)$. Since $G(A)=K_{2,3}$ has 6 edges, $k \ge 6$.
 Consequenlty $k=6$, and so $A \not \in CP(5,5)$.   \BBox
 \end{exm}   

\begin{lem}\label{le: le3.4}
For any positive integer $r=1,2,3,4$, let $\{b_i: i=1,...,r\} \subseteq R^n$ ($n\ge r$) satisfy $\seq{b_i, b_j }\ge 0$ for all $1\le i,j\le r$. 
Then there exists an orthogonal matrix $Q\in R^{n\times n}$ such that $Qb_i \in \RPlus^n$ for all $1\le i\le r$.
\end{lem}
 
\begin{proof} 
Let $B=[b_1,\ldots,b_r]\in R^{n\times r}$. Then $A=B^TB \in DN(r,r) $. But  $DN(r,r) = CP(r,r)$ for $r\le 4$ by Lemma 3.2, so  there is a nonnegative $r \times r$ matrix $C_1$ such that $A=C_1^TC_1$ (see e.g. \cite{mm}).  Now we let $C^T=[C_1^T,0]$ where the size of the 0 is $r\times (n-r)$. Thus $C\in \RPlus^{n\times r}$ and $A=B^TB =C^TC$.  By Lemma 2.1 of \cite{x}, we can find an orthogonal matrix $Q\in R^{n\times n}$ such that $C=QB$, i.e., $Qb_i \ge 0$ for all $i$.
 \end{proof} 
 
\indent  We call an $r\times n$ ($r=rank(B)\le n$) matrix $B$ a \emph{Non-Negative eQuivalent} (abbr. \emph{nnq}) if there is a
 nonsingular submatrix $B_{1}\in  \textit{R}^{r\times r}$ such that $B_{1}^{-1} B\ge 0$.  An  $n\times n$ symmetric matrix $A$ 
 with $rank(A)=r$ is called \emph{nnq} if $A$ has an $r\times n$ nnq submatrix. \\
  
\begin{thm}\label{thm: th3.5} 
Let $A \in DN(n,r)$ and $A=B^TB$ where $B\in R^{r \times n}$ has rank $r$.
 
 (i) If $r=3$ and $B$ is nnq, then $A \in CP(n,3)$.
 
 (ii) If $r=4$ and $B$ is nnq,  then $A \in CP(n,4)$.
 \end{thm}
 
 \begin{proof}
  (i) Write $B=[b_1, b_2, \ldots, b_n]$.  Since $A\in DN(n,3)$, we have $rank(B)=rank(A)=3$. Now $B$ is nnq,  we may assume w.l.o.g. that  $B_1=[b_1, b_2, b_3]$ is nonsingular and satisfies $B_1^{-1} B\ge 0$. Thus vectors $b_1, b_2, b_3$ 
 are linearly independent and $\seq{b_i, b_j}\ge 0$.  By Lemma 3.4, there exists an orthogonal matrix $Q\in R^{3\times 3}$ such  that $Qb_1, Qb_2, Qb_3 \in \RPlus^{3}$. Now we denote $P=B_1^{-1} B$, and so $QB=QB_1P \in   \RPlus^{3\times n}$.
 Consequently, $A=B^TB=B^TQ^TQB=C^TC$ where $C=QB \ge 0$, and so $3=rank(A) \le cp\mbox{-}rank(A) \le 3$,
 i.e., $A \in CP(n,3)$. 
 
 (ii)  We may assume that the submatrix $B_1=[b_1, b_2, b_3, b_4]\in \textit{R}^{4\times 4}$ is nonsingular and satisfies condition $P\equiv B_1^{-1} B\ge 0$. Then $\{b_1, b_2, b_3, b_4 \} \subseteq R^4$ are linearly independent and $\seq{b_i, b_j} \ge 0$,
 by Lemma 3.4, there exists orthogonal matrix $Q\in R^{4\times 4}$ such  that $Qb_1, Qb_2, Qb_3, Qb_4 \in \RPlus^4$.
Thus $QB=QB_1P= [Qb_1, Qb_2, Qb_3, Qb_4]  P \ge 0$.
 Consequently, $A=B^TB=B^TQ^TQB=C^TC$ where $C=QB \ge 0$, and so $4=rank(A) \le cp\mbox{-}rank(A) \le 4$,
 i.e., $A \in CP(n,4)$.  \BBox\\
\end{proof} 
 
 \indent  The argument in the proof of Theorem 3.5  implies a method for constructing CP factorizations of a matrix $A$ in $CP(n,3)$  or $CP(n,4)$ when $A$ satisfies the hypothesis. We will present an algorithm for such a factorization in this special case. \\
 \indent For a given matrix $A\in DN(n,r)$,  a \emph{symmetric rank factorization}, or briefly, a \textit{SR} factorization of $A$,  is a factorization $A=BB^T$, where $B\in \textit{R}^{n\times r}, r = rank(B) = rank(A)$. $B$ is accordingly called a \emph{SR-factor} of $A$.  The following lemma shows that if  $A\in DN(n,r)$ has a nnq SR-factor, then all the SR-factors of $A$ are nnq. This lemma will be used to show that the nnq condition is not necessary for a matrix $A\in DN(n,r)$ to be CP.

\begin{lem}\label{le: le3.6}
Let $A=B^TB=C^TC$ with $B,C\in \textit{R}^{r\times n}, r = rank(A)$ . Then $B$ is nnq if and only if $C$ is nnq. 
\end{lem}

 \begin{proof}
Write 
\[ B=[b_1, b_2, \ldots, b_n], \qquad  C=[c_1, c_2, \ldots, c_n]    \] 
Suppose that $B$ is nnq.  Then we may assume without loss of generality that $B_1=[b_1, b_2,\ldots, b_r]\in \textit{R}^{r\times r}$ is an invertible submatrix of $B$ that satisfies $B_1^{-1}B\ge 0$. We want to prove that $C$ is also nnq.  Denote 
$C_1=[c_1,c_2,\ldots, c_r]\in \textit{R}^{r\times r}$. By Lemma 3.1 there exists an $3\times 3$ orthogonal matrix $Q$ such that $C=QB$. Thus $C_1=QB_1$, which implies that $C_1$ is also invertible. Moreover, we have 
\[   C_1^{-1} C =(QB_1)^{-1}(QB) = B_1^{-1}Q^TQB =B^{-1}B \ge 0.  \]
Thus $C$ is also nnq. Conversely, if $C$ is nnq, then we can also prove that $B$ is nnq by the same argument. 
\end{proof}
 
 \indent The following example shows that the nnq condition for a matrix $A\in DN(n,3)$ to be in $CP(n,3)$ is not necessary. \\
  
\begin{exm}\label{ex: ex3.7}
Let $A=\left[ \begin{array}{cccc} 1&1&1&1\\1&2&1&2\\1&1&2&2\\1&2&2&3 \end{array} \right]$.
Then $A=C^TC$ where $C=\left[ \begin{array}{cccc} 1&1&1&1\\0&0&1&1\\0&1&0&1 \end{array} \right]$,
and so $cp\mbox{-}rank(A)=3$. Moreover
$A=B^TB$ where $B=\left[ \begin{array}{cccc} 0.6426&0.1008&0.1008&-0.4409\\
0&0.7071&-0.7071&0\\0.7662&1.2206&1.2206&1.6750 \end{array} \right]$
and no $i,j,k$ with $[b_i, b_j, b_k]^{-1}B \ge 0$. \\
\indent  Moreover, By Lemma 3.6, we know that $A$ has no nnq SR-factor. 
\end{exm}
    
\begin{algm} \label{Alg:alg3.8} A sufficient condition for checking a given matrix $A\in CP(n,3)$ and its CP factorizations.\\
\noindent \textbf{Input:} \ \ a square doubly nonnegative matrix $A\in \textit{S}_+^{n\times n}$ \\
\noindent \textbf{Output:} \ \ a nonnegative matrix $C\in \textit{R}_+^{3\times n}$ such that $A=C^TC$. 
\begin{description}
\item[Step 1.]  Use Cholesky decomposition to $A$ to get matrix $B=(b_{ij})\in \textit{R}^{r\times n}$ such that $A=B^TB$.
\item[Step 2.]  If $B\ge 0$, we are done. otherwise, we denote $B=[b_1,b_2,\ldots,b_n]$.  go to the next step.  
\item[Step 3.]  Compute determinant $d(i,j,k)=det([b_i,b_j,b_k])$ for all $(i,j,k):  1\le i<j<k\le n$. If $d(i,j,k)\neq 0$, then go to the next step. 
\item[Step 4.]  Calculate the matrix $P\equiv [b_i,b_j,b_k]^{-1} B$. If $P\ge 0$, then we have $A\in CP(n,3)$. 
\item[Step 5.]  Find an orthogonal matrix $Q_1\in \textit{R}^{3\times 3}$ such that $Q_1b_i, Q_1b_j, Q_1b_k\in \RPlus^3$. Write $Q=[Q_1b_i, Q_1b_j, Q_1b_k]$, then $Q\in \RPlus^{3\times 3}$, and $A=C^TC$ where $C=QP\in \textit{R}_+^{3\times n}$ with $rank(C)=rank(B)=3$. 
\end{description}
\end{algm} 

\indent  The orthogonal matrix $Q$ (of order three) in Step 5 can be computed by Gram-Schmidt orthogonalization. Specifically, if $B_1\equiv [\be_1,\be_2, \be_3]$ is an invertible submatrix of $B$, then we have $B_1=Q_1R$ where 
\[ R=\left[\begin{array}{ccc}
r_{11}&r_{12}&r_{13}\\
         0&r_{22}&r_{23}\\
         0&          0&r_{33}\end{array}
\right]
\]
where  $Q_1$ is an orthogonal matrix of order 3, and $r_{ii}=\sqrt{\norm{\be_i}}, r_{ij}=\seq{\be_i,\be_j}, \forall i,j\in \set{1,2,3}$. Thus $R\in \textit{R}^{3\times 3}$ is obviously a nonnegative upper triangle matrix. Now we take $Q=Q_1^T$, then we have $QB_1=R$ which satisfies the requirement in Step 5. \\
 
\begin{exm}\label{ex: ex3.9}
This example shows how to use Algorithm 3.8 to check a matrix in $CP(n,3)$ in our situation. \\
\indent  We let  
\[A=\left[ \begin{array}{ccccc} 
 1.2295 & 0.4315 & 0.0699& 0.7037&1.3685\\
 0.4315 & 0.3006 & 0.0435& 0.4241&0.7177\\
 0.0699 & 0.0435 & 0.0078& 0.0743&0.1098\\
 0.7037 &0.4241  &0.0743 & 0.7128&1.0795\\
 1.3685 &0.7177  &0.1098 & 1.0795&1.9030
\end{array}\right]
\]
It is easy to see that $A\in DN(5,3)$.  We use Matlab function EIGS to obtain its first three eigenvalues and the corresponding eigenvectors, which form a rank-3 decomposition $A=B^TB$ where 
\[  B=\left[ \begin{array}{ccccc}
1.0272 & 0.5025 & 0.0795 & 0.7824 &1.3729\\
-0.4151 & 0.1924 & 0.0309 & 0.2604 &0.0900\\
-0.0450 & 0.1053 &-0.0224 &-0.1814 &0.0998
\end{array}\right]
\] 
Now we choose $B_1:=B[:, 1:3]$, i.e., $B_1$ is the submatrix of $B$ consisting of the first three columns of $B$.  We find that 
\[ B_1^{-1}B= \left[ \begin{array}{ccccc}
 1& 0 & 0& 0.0526 &0.5396\\
 0& 1 & 0& 0.1055 &1.4368\\
 0& 0 & 1& 8.4895 &1.2159
\end{array}\right]
\]
It follows that $B$ is an nnq matrix, thus we have $A\in CP(5,3)$. In fact, we have CP decomposition $A=C^TC$ with 
\[ C =  \left[ \begin{array}{ccccc}
0.1278 & 0.4274&0.0587 &0.5505 &0.7545\\
0.6399 & 0.1758&0.0602 &0.5631 & 0.6711\\
0.8965 & 0.2949&0.0267 &0.3046 & 0.9398
\end{array}\right]
\]
which implies that $A\in CP(5,3)$ since $3=rank(A)\le \textit{cp-rank}(A)\le 3$. 
\end{exm}

\indent  Denote by $\textit{C}_A$ the convex cone generated by the columns of matrix $A$, $Extr(C)$ for the set of all
 extreme rays of a cone $C$, and denote $c_M\equiv \abs{Extr(\textit{C}_M)}$ for any matrix $M$ .  It is shown in \cite{x} that $c_A =c_B$  where $A=B^TB$.  \\
 \indent  The following theorem deals with the cases when the number of extreme rays of cone generated by the columns of $A$ 
 equals the rank (and the CP-rank) of $A$.
 \begin{lem}\label{le: le3.10}
 \begin{description}
 \item[(1)]  Let $A\in CP(n,r)$ where $r\le 4$. Then $\abs{Extr(\textit{C}_A)}= r$. 
 \item[(2)]  Let $A\in DN(n,r)$ with $m=\abs{Extr(\textit{C}_A)}\le 4$. Then $A\in CP_n$ with $\cprk(A)\le m$.. 
 \end{description}
 \end{lem}
 
 \begin{proof}
(1).\ \  Suppose that $A=[A_1,\ldots,A_n]=BB^T$, where $B=[\be_1,\ldots,\be_r]\in \textit{R}_+^{n\times r}$ with $r=rank(A)=\cprk(A)\le 4$.  Denote $m=\abs{Extr(\textit{C}_A)}$. Then we have $m=\abs{Extr(\textit{C}_B)}\le r$. Suppose that $m<r$.  We may assume without loss of generality that $Ext(C_B)=\set{\be_1,\be_2,\ldots,\be_m}$,  and denote $B_1=[\be_1,\be_2,\ldots,\be_m]\in \RPlus^{n\times m}$. By definition of the extreme rays, there exists some nonnegative matrix 
$X_1\in  \RPlus^{m\times r}$ such that $B=B_1X_1$. Now we consider matrix $Y=X_1X_1^T\in  DN_{m}=CP_m$ (note that $m\le 4$). Then there exists $X_2\in  \RPlus^{m\times m}$ such that $Y=X_1X_1^T=X_2X_2^T$. Thus 
\[ A =BB^T =B_1(X_1X_1^T) B_1^T =B_1(X_2X_2^T) B_1^T =B_2B_2^T  
\] 
where $B_2=B_1X_2\in \RPlus^{n\times m}$. It follows that $r=\cprk(A)\le m$, a contradiction. Consequently we have $m=r$. 
(2).\ \  Let $A\in DN(n,r)$, $m\le 4$, and denote $E\equiv Extr(\textit{C}_A)$. We may assume w.l.o.g. that $E=\set{\al_1,\ldots,\al_m}$ where $\al_j$ is the $j$th column of $A$.  Now we denote $A_1=[\al_1,\ldots,\al_m]\in \textit{R}^{n\times m}$. By the definition of $E$ there exists a nonnegative matrix $W\in \textit{R}_+^{m\times n}$ such 
that $A=A_1W$.  Now we write $A_1$ in block form 
\[  A_1=\left[\begin{array}{c} A_{11}\\ A_{21} \end{array} \right],  \]
where $A_{11}\in \textit{R}^{m\times m}$, then $A$ has the blocking form  
\begin{equation}\label{eq:blockA}
A =\left[\begin{array}{cc} A_{11} & A_{11}X\\ 
                                                  X^TA_{11} & X^TA_{11}X\end{array} \right]
\end{equation}
due to the symmetry of $A$, where $X\in  \RPlus^{m\times (n-m)}$.  It follows from the fact $A\in DN(n,r)$ that $A_{11}\in DN(m,r)$ where $m\le 4$. Thus we have $A_{11}\in CP_m$ since $CP_m=DN_m$ for $m\le 4$. Thus there exists some nonnegative $m\times m_1$ matrix $B_1$ ($r\le m_1=\cprk(A_{11})\le m$) such that $A_{11}=B_1B_1^T$.  Now we choose 
$B = [B_1^T, B_1^TX]^T$, then we have $A=BB^T$ with $B\in \RPlus^{n\times m_1}$. Consequently we have   
$A\in CP_n$ with $\cprk(A)=m_1\le m$. 
\end{proof}

\indent  Given a matrix $A\in CP(n,3)$, we have $\abs{Extr(\textit{C}_A)}\ge 3$ by Lemma 3.10. But $A$ may not be a \emph{nnq} matrix.  However, if  we have $\abs{Extr(\textit{C}_A)}=3$ , then $A$ must be nnq (and thus CP). Specifically, we have 
 \begin{thm}\label{th: thm3.11}
Let $A\in DN(n,3)$ with $\abs{Extr(\textit{C}_A)}=3$.  Then $A\in CP(n,3)$ if and only if $A$ is nnq.
 \end{thm}
 
 \begin{proof}
 We write $A=Gram(\al_1,\al_2,\ldots,\al_n)\in DN(n,3)$ where $\al_j\in \textit{R}^{3}$.  Then the theorem states that when the cone $\textit{C}_A$ has three extreme rays, we have $A\in CP(n,3)$ if and only if  there exists a $3\times 3$ nonsingular principal submatrix of $A$, denoted by $A_{11}$,  such that $A_{11}^{-1}A_{1}\ge 0$ where $A_{1}$ is the $3\times n$ submatrix of $A$ consisting of the three rows of $A$ corresponding to those of $A_{11}$'s.\\
 For convenience, we may assume that all diagonal elements of $A$ are positive since otherwise $A$ can be reduced to a lower order.  Furthermore, we can assume that all the diagonal elements of $A$ are ones since otherwise we can substitute $A$ by $D^{-1/2}AD^{-1/2}$ which has the same CP property with $A$ in our interest. \\
 \indent For the necessary, we assume that $A\in CP(n,3)$.  Write $A$ as $A=B^TB$ with $B=[\al_1,\al_2,\ldots,\al_n]\in \textit{R}^{3\times n}, rank(B)=3$.  
 
 Denote $\Gamma =\set{\al_1,\al_2,\ldots,\al_n}\subset \textit{R}_+^3$. Then there exists an orthogonal matrix $Q$ of order three such that $Q\al_1,Q\al_2,\ldots,Q\al_n\in \textit{R}_+^3$. Denote $\be_j=Q\al_j, \forall j=1,2,...,n$. Then we have $A=Gram(\be_1,\ldots,\be_n)$.  Since $\textit{cp-rank}(A)=rank(A)=3$, we know that $A$ must be an interior point in the cone $CP(n,3)$. Now we denote $C_A=\textit{conv}\set{\be_1,\be_2,\ldots,\be_n}$. It follows from Lemma 3.4 that $\textit{cp-rank}(A)$ is exactly the maximal number of conically independent vectors (called \emph{con-rank}) among $\Gamma$ . Note that a set of vectors $\Gamma$ is called conically dependent if there exists a vector $\al_j$ in $\Gamma$ such that $\al_j$ can be represented as a nonnegative linear combination, and called conically independent if $\Gamma$ is not conically dependent. Since $\abs{Extr(\textit{C}_A)}=3$, there exist three conically independent vectors, say, $\al_1,\al_2,\al_3\in \Gamma$, such that any other vector in $\Gamma$ can be expressed as a nonnegative linear combination of $\set{\al_1,\al_2,\al_3}$.  In fact, the three vectors $\al_1,\al_2,\al_3$ are also linearly independent. \\
 \indent  For each $j\in \set{1,2,...,n}$, there exists a nonnegative optimal solution $(x_0,y_0,z_0)$ to minimize the function 
 \begin{equation}
 f_j(x,y,z)=\norm{\al_j-x\al_1-y\al_2-z\al_3}
 \end{equation} 
with $f_j(x_0,y_0,z_0)=\min f_j(x,y,z) =0$. Since the minimization of function $f$ is equivalent to that of  function $F_j(x,y,z)\equiv f_j(x,y,z)^2$, i.e., 
 \[ \min_{x,y,z\ge 0} F_j(x,y,z) = \seq{\al_j-x\al_1-y\al_2-z\al_3, \al_j-x\al_1-y\al_2-z\al_3}
  \] 
By taking the gradient of $F$  (notice that $\seq{\al_i,\al_j}=a_{ij}$ and $a_{ii}=1$ for all $i,j\in \set{1,2,...,n}$) to be zero, we get 
 \begin{eqnarray}
              2x+2a_{12}y+2a_{13}z =&2a_{1j} \\
 2a_{12}x+             2y+2a_{23}z= &2a_{2j} \\
 2a_{13}x+ 2a_{23}y+            2z= &2a_{3j}   
 \end{eqnarray}
 Thus we get 
 \[ \left[ \begin{array}{c}
 x\\  y\\ z \end{array}\right] = \left[ \begin{array}{ccc}
1  & a_{12} & a_{13}\\  
a_{12}  & 1 & a_{23}\\ 
a_{13}  & a_{23} & 1\end{array}\right]^{-1} \left[ \begin{array}{c}
 a_{1j}\\  a_{2j}\\ a_{3j} \end{array}\right]  
 \]
which is equivalent to $A_{11}^{-1}A_{1}\ge 0$ if we take $A_{11}=A([1,2,3],[1,2,3])$ and $A_{1}=A([1,2,3],:)$ .\\
 \indent Now we come to prove the sufficiency.  Suppose that $A\in DN(n,3)$ and that  
 $A_{11}^{-1}A_1\ge 0$ for $A_1=A([i,j,k],:), A_{11}=A([i,j,k],[i,j,k])$ for some positive integers $i,j,k$ ($1\le i<j<k\le n$), i.e.,
 $A_{11}$ is a submatrix of $A$ whose columns and rows are both indexed by $\set{i,j,k}$ (here and afterwards we abuse
MATLAB notations when the matrix indexing is concerned) . We may assume w.l.o.g. that $A_{11}=A(1:3,1:3)$ such that
$rank(A_{11})=3$ and $A_{11}^{-1}A_1\ge 0$ where $A_1 = A(1:3, :)$. We need
to prove that matrix $A$ is a CP matrix with $\textit{cp-rank}(A)=rank(A)=3$. \\
\indent For this purpose, we write $A$ into the following blocking form:
\begin{equation}\label{eq: blockA}
\left[\begin{array}{cc}
A_{11} & A_{12}\\
A_{21} &A_{22}
\end{array}\right]
\end{equation}
Then $A_1=[A_{11}, A_{12}]$ and $A_{11}^{-1}A_1=[I_3, A_{11}^{-1}A_{12}]\in \RPlus^{3\times n}$. Denote
$A_3=A_{11}^{-1}A_{12}$. Then we have $A_3\in \RPlus^{3\times (n-3)}$ and $A_{12}=A_{11}A_{3}$. By the symmetry of
$A$,  we get $A_{21}=A_{12}^T = A_3^T A_{11}$. Now that $A_{11}\in DN(3,3)=CP(3,3)$, so there is a nonnegative $3\times 3$ matrix $B_1$ such that $A_{11}=B_1^TB_1$.  Denote $B=[B_1, B_1A_3]$, then $B\in \RPlus^{3\times n}$ and $A=B^TB$ which implies that $A\in CP(n,3)$. The proof is completed.    
\end{proof}

\indent   Theorem 3.11 can be used to determine when a matrix $A\in DN(n,3)$ is in $CP(n,3)$ when the cone $\textit{C}_A$ has three extreme rays. The following example illustrates the effectiveness of  Theorem 3.11. \\
\begin{exm}\label{ex: ex3.12}
Consider matrix $A$ in Example 3.10.  We let $A_{11}=A(1:3,1:3)$, and $A_1=A(1:3, :)$. Then we have  
\begin{equation}\label{exm: ex12}
A_{11}^{-1}A_1 = \left[\begin{array}{ccccc}
1  & 0  & 0 & 0.0538 & 0.5388\\
0  & 1  & 0 & 0.1291 & 1.4292\\
0  & 0 &  1 & 8.3229 & 1.2776 
\end{array}\right]
\end{equation}
Thus we once again confirm that $A\in CP(5,3)$ by Theorem 3.11.  
 \end{exm}

%
 \indent The characterization of $CP(n,3)$ is not over as we can see from Example 3.8 that the condition in Theorem 3.11 cannot be removed. However, there are some cases for $A\in CP(n,3)$ when $\textit{C}_A$ has more than three extreme rays. Actually we can find an arbitrary number of extreme rays for  $\textit{C}_A$.  For example, in Example 3.8, there are five extreme rays in $\textit{C}_A$. 

\noindent \textbf{Acknowledgement.} The authors would like to thank Professor Jor-Ting Chan for providing
Example 2.2.

\end{document}